\documentclass[11pt]{article}
\usepackage{graphics}
\usepackage{epsfig}
\usepackage{amsmath}
\usepackage{amssymb}
\usepackage{latexsym}
\usepackage{amscd}
\usepackage{psfrag}
\usepackage{rotating}
\usepackage{amsfonts}
\usepackage{amsthm}
\usepackage{float}
\newtheorem{theorem}{Theorem}[section]
\newtheorem{lemma}{Lemma}[section]
\newtheorem{remark}{Remark}[section]

\newtheorem{example}{Example}[section]
\newtheorem{proposition}{Proposition}[section]

\begin{document}

\title{The structure of mirrors on Platonic surfaces}
\bigskip\bigskip\bigskip
\author{Adnan Meleko\u glu\\ Department of Mathematics\\ Faculty of Arts and Sciences \\
Adnan Menderes University \\ 09010 Ayd\i n, TURKEY\\
email: amelekoglu@adu.edu.tr \and
David Singerman\\ School of Mathematics\\ University of Southampton\\
 Southampton, SO17 1BJ\\ ENGLAND\\ email:
d.singerman@soton.ac.uk }
\date{}

\maketitle

\begin{abstract}

A Platonic surface is a Riemann surface that underlies a regular map and so we can consider its vertices, edge-centres and face-centres. A symmetry (anticonformal involution) of the surface will fix a number of simple closed curves which we call mirrors. These mirrors must pass through the vertices, edge-centres and face-centres in some sequence which we call the pattern of the mirror. Here we investigate these patterns for various well-known families of Platonic surfaces, including genus 1 regular maps,  and regular maps on  Hurwitz surfaces and Fermat curves. The genesis of this paper is classical. Klein in Section 13 of his famous 1878 paper \cite{Kle}, worked out the pattern of the mirrors on the Klein quartic and Coxeter in his book on regular polytopes worked out the patterns for mirrors on the regular solids. We believe that this topic has not been pursued since then.

\end{abstract}



\section{Introduction}
\setcounter{equation}{0}

A compact Riemann surface $X$ is \emph{Platonic} if it admits a
regular map. By a map (or clean dessin d'enfant) on $X$ we mean an
embedding of a graph $\mathcal{G}$ into $X$ such that $X\setminus
\mathcal G$ is a union of simply connected polygonal regions,
called faces. A  map thus has vertices, edges and faces. A
directed edge is called a \emph{dart} and a map is called
\emph{regular} if its automorphism group acts transitively on its
darts.  The Platonic solids are the most well-known examples of
regular maps. These lie on the Riemann sphere.  To be more
precise, we recall how we study maps using triangle groups. The
\emph{universal map of type} $\{m,n\}$ is the tessellation of one
of the three simply connected Riemann surfaces, $\mathcal{U}$, the
Riemann sphere $\Sigma$, the Euclidean plane $\mathbb{C}$, or the
hyperbolic plane, $\mathbb{H}$ depending on whether the genus of
the surface is 0, 1, or $>1$ by regular $m$-gons with $n$ meeting
at each vertex. This map is denoted by  $\hat{\mathcal M}(m,n)$.
The automorphism group,  and also the conformal automorphism group
of $\hat{\mathcal M}(m,n)$  is the triangle group $\Gamma[2,m,n]$.
This is the group with presentation
\begin{eqnarray*}
\langle A,B,C \mid A^2=B^m=C^n=ABC=1\rangle,
\end{eqnarray*}
 the orientation preserving subgroup of the extended triangle group $\Gamma(2,m,n)$
 generated by three reflections $P$, $Q$, $R$ as described in \S 2.

 In general, a map is of \emph{type} $\{m,n\}$ if $m$ is the lcm of the face sizes and $n$ is the lcm of the vertex valencies. As shown in \cite{JS}, every map of type $\{m,n\}$ is  a quotient of $\hat{\mathcal M}(m,n)$ by a subgroup $M$ of the triangle group $\Gamma[2,m,n]$. Then  $M$ is called a \emph{map subgroup} of $\hat{\mathcal M}(m,n)$ or sometimes a fundamental group of $\hat{\mathcal M}(m,n)$, inside $\Gamma[2,m,n]$. Thus, associated with a map $\mathcal M$ we have a unique Riemann surface $\mathcal{U}/M$. The map is regular if and only if $M$ is a normal subgroup of $\Gamma[2,m,n]$. Thus, a Platonic surface is one of the form $\mathcal{U}/M$ where $M$ is a normal subgroup of a triangle group and $\mathcal{U}$ is a simply connected Riemann surface.  A \emph{symmetry} of a Riemann surface $X$ is an anticonformal involution $T\colon X\rightarrow X$ and a \emph{symmetric} Riemann  surface is one that admits a symmetry, for example the Riemann sphere  admits complex conjugation as a symmetry.

We are interested in the fixed point set of a symmetry. This is given by a classical theorem of Harnack which tells us that the fixed point set of a symmetry $T$ consists of a number $k$ of disjoint Jordan curves, where $0\le k\le g+1$ and $g$ is the genus of $X$. Each such Jordan curve is called a \emph{mirror} of $T$.

On a regular map $\mathcal M$ a special role is played by the
fixed points of automorphisms of the map (or the Riemann surface).
These are the vertices, edge-centres and face-centres of $\mathcal
M$. We call such points the \emph{geometric points} of $\mathcal
M$. As an automorphism preserves the vertices, edge-centres and
face-centres, it follows that a mirror must pass through vertices,
edge-centres and face-centres. Let us denote a vertex by
$\mathbf{0}$, an edge-centre by $\mathbf{1}$ and a face-centre by
$\mathbf{2}$. This notation comes from Coxeter~\cite{Co}. If we
travel along a mirror we will in succession meet a sequence of
geometric points of the map, and if we have a map on a compact
surface then this sequence is finite. This sequence is called the
\emph{pattern} of the mirror. We sometimes use the term pattern of
a symmetry if the symmetry fixes a unique mirror. In \cite{Co},
Section 4.5 the patterns of the mirrors of the regular solids are
computed and displayed here in Table~\ref{patternspherical} in
\S2.

As an example consider the dodecahedral map on the Riemann sphere in Figure~\ref{Dodecahedron}. The reflection $z\mapsto \bar z$ fixes the equator which passes through vertices, edge-centres and face-centres giving the patten $\mathbf{010212010212}$, which we abbreviate to $(\mathbf{010212})^2$, see Figure~\ref{Dodecahedron}.

\begin{figure}[htb]
  \centering
  \includegraphics[width=0.3\textwidth]{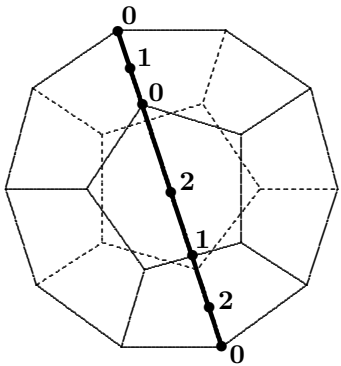}
  \caption{Dodecahedron}
  \label{Dodecahedron}
\end{figure}

 However well before Coxeter,  Felix Klein had discussed this very idea in his famous paper where he introduced the Klein quartic. This corresponds to his  Riemann surface of genus 3 with automorphism group PSL(2,7) of order 168.  Associated with this surface is the well-known map of type $\{3,7\}$. The Klein surface is symmetric and in Section 13 of his paper Klein points out that on his surface a mirror has pattern
$(\mathbf{010212})^3$. (However in Klein's paper (see \cite{Kle}, p465 or \cite{Levy}, p322) he uses $a,b,c$ for $\mathbf{2,0,1}$.)

Our aim is to discuss patterns for Platonic surfaces in general.

\section{Patterns of mirrors}
\setcounter{equation}{0}

An automorphism of a Riemann surface $X$ is a biholomorphic homeomorphism
$T\colon X\rightarrow X$. It is known that if we have a map $\mathcal M$ on a surface
$X$, then this induces a natural complex structure on $X$ that makes $X$ into a Riemann surface.
(See \cite{Grot, JS, Si1976}).
An automorphism of the map $\mathcal M$ is a automorphism of the Riemann surface
$X$ that transforms the map to itself and preserves incidence.  We let $Aut X$,
$Aut^\pm X$ denote the conformal automorphism group of $X$, and the extended
automorphism group of $X$, respectively, where the extended automorphism group allows
anticonformal automorphisms. Similarly, we let $Aut\mathcal M$ and $Aut^\pm\mathcal{M}$
denote the sense preserving automorphism group of $\mathcal M$ and the full automorphism group of
$\mathcal M$, respectively. We have $Aut \mathcal M$ lies in $Aut X$ with a similar remark
about the full groups. In most cases, these groups coincide although there are cases where
the automorphism group of the Riemann surface is strictly larger than the automorphism group
of the map. A  regular map is called \emph{reflexible} if it admits an
orientaion-reversing automorphism, (often this will be a reflection) and then
$|Aut^\pm\mathcal M|=2|Aut\mathcal M|.$

Let $X$ be a Riemann surface of genus $g$ and let $\mathcal{M}$
be a regular map of type $\{m,n\}$ on $X$. Let $F$ be a face of
$\mathcal{M}$. If we join the centre of $F$ to the centres of the
edges and the vertices surrounding $F$ by geodesic arcs, we obtain
a subdivision of $F$ into $2m$ triangles. Each triangle has angles
$\pi/2$, $\pi/m$ and $\pi/n$, and will be called a
\emph{$(2,m,n)$-triangle}. In this way, we obtain a triangulation
of $X$. Note that there are as many $(2,m,n)$-triangles as the
order of $Aut^\pm\mathcal{M}$, and the reflexibility of
$\mathcal{M}$ implies that $Aut^\pm\mathcal{M}$ is transitive on
these triangles.

Let $T$ be a $(2,m,n)$-triangle on $X$ and let $P$, $Q$ and $R$
denote the reflections in the sides of $T$, as indicated in Figure~\ref{Figuretriangle}. These
reflections satisfy the relations
\begin{eqnarray}
P^2=Q^2=R^2=(PQ)^2=(QR)^m=(RP)^n=1. \label{presentation2mn}
\end{eqnarray}
The triangle group $\Gamma[2,m,n]$ is generated by $A=PQ$, $B=QR$ and $C=RP$ obeying the relations
\[
A^2=B^m=C^n=ABC=1.
\]
The reflections $P$, $Q$ and $R$ generate $Aut^\pm\mathcal{M}$.
However, the relations in (\ref{presentation2mn}) only give a finite
group if $g=0$ and hence in other cases we need at least one more
relation to get a presentation for $Aut^\pm\mathcal{M}$. The pair of any
two successive geometric points on a mirror is either $\mathbf{01}$,
$\mathbf{02}$ or $\mathbf{12}$ (or in reverse order). Let $T^\ast$ be
a $(2,m,n)$-triangle. Then we see that each corner of $T^\ast$
either is a vertex, a face-center or an edge-center of
$\mathcal{M}$. So we can label the corners of $T^\ast$ with
$\mathbf{0}$, $\mathbf{1}$ and $\mathbf{2}$. Each of the pairs
$\mathbf{01}$, $\mathbf{02}$ and $\mathbf{12}$ corresponds to one of the
sides of $T^\ast$. We will call the corresponding sides of
$T^\ast$ the $\mathbf{01}$-\emph{side}, the $\mathbf{02}$-\emph{side}
and the $\mathbf{12}$-\emph{side}. See Figure~\ref{Figuretriangle}.

\begin{figure}[htb]
  \centering
  \includegraphics[width=0.35\textwidth]{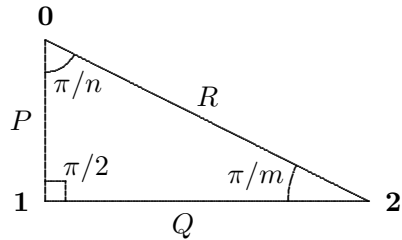}
  \caption{$(2,m,n)$-triangle}
  \label{Figuretriangle}
\end{figure}

Let $M$ be a mirror on $X$ fixed by a reflection. Then $M$ passes
through some geometric points of $\mathcal{M}$. These geometric points
form a periodic sequence of the form
\begin{eqnarray}
\underbrace{a_1a_2\dots a_{k-1}a_k}_{1}\, \underbrace{a_1a_2\dots
a_{k-1}a_k}_{2}\,\dots \underbrace{a_1a_2\dots a_{k-1}a_k}_{N}
\label{pattern}
\end{eqnarray}
which we call the \emph{pattern} of $M$, where $a_i\in
\{\mathbf{0},\mathbf{1},\mathbf{2}\}$ and $1\leq i\leq k$. We call
 $a_1a_2\dots a_{k-1}a_k$ of (\ref{pattern}) a
\emph{link} of the pattern and call $N$ the \emph{link index}.
 We abbreviate the pattern (\ref{pattern})  to $(a_1a_2\dots
a_{k-1}a_k)^N$. For example, from the introduction we see that for
the dodecahedron the links are all $(\mathbf{010212})^2$ so the
link index is 2. For the Klein map we have the same link but now
the link index is 3. As Klein thought of his map as being a higher
genus version of the dodecahedron it is interesting to note the
similarities in their patterns!

\emph{As a Riemann surface determines a unique map, we can talk
about a pattern of a Riemann surface. Note that the dual map is
also associated with this Riemann surface but the patterns for the
dual are found by interchanging $\mathbf{0}$s and $\mathbf{2}$s. }

We display the patterns for the Platonic solids (the regular maps on the sphere) in Table~\ref{patternspherical}. Most of this is from Coxeter \cite{Co} \S4.5, except for the easy cases of the dihedron and its dual hosohedron. Note that in Table~\ref{patternspherical},  the tetrahedron, type $\{3,3\}$ or the icosahedron, type $\{3,5\}$
have only one pattern listed but for the octahedron, type $\{3,4\}$ there are two patterns listed. We shall see why soon.

\begin{table}[H]
\begin{center}
\caption{\label{patternspherical} Spherical Maps and Patterns}
\begin{tabular}{|l|c|c|l|}
  \hline
 Platonic Solid  &   Map Type & Number of Mirrors & Pattern \\
  \hline\hline
  \cline{1-4} Tetrahedron  & $$\{3,3\}$$ & 6 &$\mathbf{010212}$ \\

  \hline
\cline{1-4} &  & 3& $(\mathbf{01})^4$ \\

    \cline{3-4} \raisebox{1.5ex}[0pt]{Octahedron}  & \raisebox{1.5ex}[0pt]{$\{3,4\}$} &6&$(\mathbf{0212})^2$
    \\

    \hline
\cline{1-4} &  & 3& $(\mathbf{12})^4$ \\

    \cline{3-4} \raisebox{1.5ex}[0pt]{Cube}  & \raisebox{1.5ex}[0pt]{$\{4,3\}$} &6&$(\mathbf{0102})^2$
    \\

    \hline
\cline{1-4} Icosahedron  & $\{3,5\}$ & 15 &$(\mathbf{010212})^2$ \\

\hline
\cline{1-4} Dodecahedron  & $\{5,3\}$ & 15 &$(\mathbf{010212})^2$ \\ \hline

\hline
\cline{1-4} &  & 1& $(\mathbf{12})^n$ \\

    \cline{3-4} \raisebox{1.5ex}[0pt]{Hosohedron}  & \raisebox{1.5ex}[0pt]{$\{2,n\}$  \ $n$ odd} & $n$ & $\mathbf{0102}$
    \\\hline

    \cline{1-4} &  & $n/2$ & $(\mathbf{01})^2$ \\

     \cline{3-4} Hosohedron & $\{2,n\}$  \ $n$ even & $n/2$ & $(\mathbf{02})^2$ \\

      \cline{3-4} &  &  1 & $(\mathbf{12})^n$ \\
\hline

   \hline
\cline{1-4} &  & 1& $(\mathbf{01})^n$ \\

    \cline{3-4} \raisebox{1.5ex}[0pt]{Dihedron}  & \raisebox{1.5ex}[0pt]{$\{n,2\}$  \ $n$ odd} & $n$ & $\mathbf{0212}$
    \\\hline

    \cline{1-4} &  & 1& $(\mathbf{01})^n$ \\

     \cline{3-4} Dihedron & $\{n,2\}$  \ $n$ even & $n/2$ & $(\mathbf{02})^2$ \\

      \cline{3-4} &  &  $n/2$ & $(\mathbf{12})^2$ \\
\hline

\end{tabular}
\end{center}
\end{table}

\section{Universal maps}
\setcounter{equation}{0}

\subsection{Patterns of universal maps}

We now find the patterns for the universal maps $\hat{\mathcal M}(m,n)$.
The patterns for the universal maps depend just on the parity of $m,n$.
We suppose that our universal maps are infinite, that is, for now
we ignore the spherical maps.

We give these patterns in the following result.

\begin{theorem}
\label{theorempattern}
The patterns for the nonspherical universal maps are given in
\hbox{Table~\ref{patternuniversal}}.
\end{theorem}

\begin{table}[H]
\begin{center}
\caption{\label{patternuniversal}Patterns of Mirrors of universal maps}
%\smallskip
\begin{tabular}{|c|l|l|}
   \hline
Case & Reflections & Pattern \\
   \hline\hline
   \cline{1-3} $m$ and $n$ odd & $P$, $Q$, $R$ &$(\mathbf{010212})^\infty$
\\ \hline
\cline{1-3} & $P$& $(\mathbf{01})^\infty$ \\
     \cline{2-3}  \raisebox{1.5ex}[0pt]{$m$ odd $n$ even} &$Q$,
$R$&$(\mathbf{0212})^\infty$
     \\ \hline
      \cline{1-3}               &$P$& $(\mathbf{01})^\infty$ \\
\cline{2-3}  $m$ and $n$ even  & $Q$&$(\mathbf{12})^\infty$ \\
     \cline{2-3}                            &$R$&$(\mathbf{02})^\infty$
\\\hline
  \cline{1-3}   & $P$, $R$&$(\mathbf{0102})^\infty$\\
       \cline{2-3} \raisebox{1.5ex}[0pt]{$m$ even $n$ odd}&$Q$&
$(\mathbf{12})^\infty$ \\ \hline
\end{tabular}
\end{center}
\end{table}

\begin{proof}
\begin{enumerate}
\item $m$, $n$ even.  There are mirrors joining opposite vertices of a $n$-sided polygon that pass through face-centres. These correspond to the $R$-symmetries and so the pattern is $(\mathbf{02})^\infty.$ Similarly the $Q$-symmetries pass though edge-centres and face-centres giving their pattern to be $(\mathbf{12})^\infty$. The $P$-symmetries pass through edge-centres and vertices giving the pattern $(\mathbf{01})^\infty.$ (Take the square lattice as an example.)

\item $n$ even, $m$ odd. Now the $R$-symmetries are through lines joining opposite vertices that pass through a face-centre and then continue along an edge of the adjacent polygon, giving its pattern as $(\mathbf{0212})^\infty$. The pattern for the $Q$-symmetries is the same, but the $P$-symmetries just go through edge-centres and vertices and so the pattern is $(\mathbf{01})^\infty$. (An example here might be tessellation of the plane by equilateral triangles.)

\item $n$ odd, $m$ even. We can just take the dual of the above situation. This interchanges $\mathbf{0}$ and $\mathbf{2}$ and also $P$ and $Q$.  Hence the pattern for the $P$ and $R$ symmetries is $(\mathbf{0102})^\infty$ and the pattern for the $Q$-symmetries is $(\mathbf{12})^\infty.$
    (An example here might be the hexagonal lattice.)

\item $m$, $n$ odd. Now an edge of the map goes from a vertex to an edge-centre and then to a vertex and  (because $m$ is odd, to a face-centre) and then (because $n$ is odd) to the opposite edge-centre and then to a face-centre again giving the pattern $(\mathbf{010212})^\infty$. (As an example, consider the icosahedron though this is a finite map.)
\end{enumerate}
\end{proof}

It follows from the above discussions that the pattern of a mirror line is obtained from
one of the following links: $\mathbf{12}$, $\mathbf{02}$,
$\mathbf{01}$, $\mathbf{0102}$, $\mathbf{0212}$, $\mathbf{010212}$.

\subsection{Mirror automorphisms}

Let $M$ be a mirror of a map $\mathcal M$. A \emph{mirror
automorphism} of $M$ is an orientation-preserving automorphism of
the map $M$ that fixes $M$ and has no fixed points on $M$. Note
that the inverse of a mirror automorphism is also a mirror
automorphism. Suppose that $M$ contains vertices. (Similar
arguments apply if it contains edge-centres or face-centres.) If
$\alpha^{\prime}$ is a mirror automorphism of $M$ and if $v_1$ is
a vertex of $M$ then $\alpha^{\prime}v_1$ is another vertex of
$M$. Let $\alpha$ be a mirror automorphism that takes $v_1$ to the
closest vertex which lies on $M$. Then $\alpha$ generates the
cyclic group of all mirror automorphisms of $M$.

\subsection{Mirror automorphisms of universal maps}

We now determine the possible mirror automorphisms of $\hat{\mathcal M}(m,n)$ in
terms of the generators $P$,$Q$,$R$ of $\Gamma(2,m,n)$.  Let $M$ be a mirror with
pattern $(\mathbf{12})^\infty$.
As we saw earlier, this pattern occurs in the cases where $m$ is even.
Let $F_1$ and $F_2$ be two faces of $\hat{\mathcal M}(m,n)$ with a common edge
such that $M$ passes through the centre of the common edge and the centres
of $F_1$ and $F_2$. Let us divide $F_1$ and $F_2$ into $2m$ $(2,m,n)$-triangles
as in  \S 2. Let $T_1$ be one of these triangles contained in $F_1$ such that its
$\mathbf{12}$-side lies on $M$. Let $P$, $Q$ and $R$ be the reflections in the sides
of $T_1$. These reflections generate $Aut^\pm\hat{\mathcal M}(m,n)$
and satisfy (\ref{presentation2mn}). The automorphism
$(RQ)^{({\frac{m}{2}-1)}}R$ is a reflection of $\hat{\mathcal M}(m,n)$ that maps $T_1$ onto
another triangle, say $T_2$, in $F_1$ such that the
$\mathbf{12}$-side of $T_2$ lies on $M$, and $P$ maps $T_2$
onto another triangle, say $T_3$, in $F_2$. Therefore,
$(RQ)^{({\frac{m}{2}-1)}}RP$ maps $T_1$ to $T_3$ and hence $F_1$
to $F_2$. It is not difficult to see that
$(RQ)^{({\frac{m}{2}-1)}}RP$ fixes $M$ setwise and
cyclically permutes the links of the pattern of $M$. Therefore,
$(RQ)^{({\frac{m}{2}-1)}}RP$ is the mirror automorphism of $M$.
See Figure~\ref{Figuremirroraut}, where $n=3$ and $m=6$.

Similarly, we can determine the mirror automorphisms corresponding to the other
patterns and we display the results in Table~\ref{tablemirroraut}, both in terms of $P$,$Q$,$R$ and
$A$,$B$,$C$. In Table~\ref{tablemirroraut}, for each pattern we give only a link.

\begin{figure}[htb]
\centering
\includegraphics[width=0.6\textwidth]{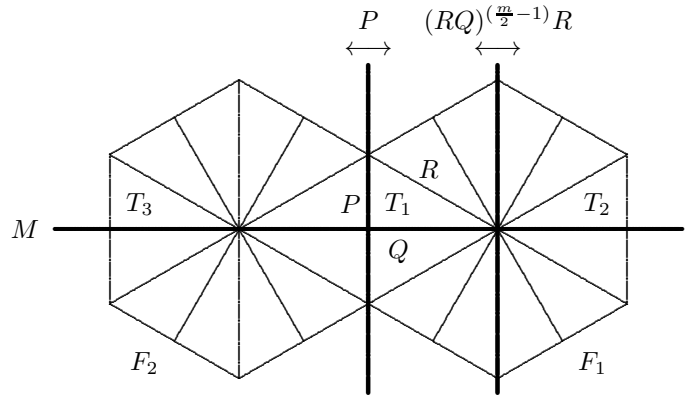}
\caption{Mirror automorphism corresponding to the pattern $(\mathbf{12})^\infty$}
\label{Figuremirroraut}
\end{figure}

\begin{remark}
\label{remarkconjugacy}
\emph{Note that in the above discussion the reflections $(RQ)^{({\frac{m}{2}-1)}}R$ and  $P$
have been chosen in order that their mirror lines are minimum distance apart
and orthogonal to $M$. So the product of these reflections is the mirror
automorphism of $M$.
If $M^{\prime}$ is another mirror line with the same pattern then there exists
an orientation preserving automorphism $f$ of $\hat{\mathcal M}(m,n)$ such that $f$ maps $M$ onto
$M^{\prime}$ since $Aut^\pm\hat{\mathcal M}(m,n)$ is transitive on $(2,m,n)$-triangles.
Therefore, if $\mu$ is the mirror automorphism of $M$, then $\mu^{\prime}=f\mu f^{-1}$
is the mirror automorphism of $M^{\prime}$. As a result, each pattern corresponds
to a conjugacy class of mirror automorphisms. This argument applies to all maps.}
\end{remark}

\begin{table}[H]
\begin{center}
\caption{\label{tablemirroraut}Patterns and Mirror Automorphisms}
\begin{tabular}{|c|l|l|l|}
  \hline
  Case & Link      & Mirror Automorphism  & Mirror Automorphism \\
  \hline\hline
     1 & $\mathbf{01}$     & $(RP)^{\frac{n}{2}-1}RQ$                      & $C^{\frac{n}{2}}A$  \\\hline
     2 & $\mathbf{02}$     & $(QR)^{\frac{m}{2}-1}Q(PR)^{\frac{n}{2}-1}P$  & $B^{\frac{m}{2}}C^{\frac{n}{2}}$  \\\hline
     3 & $\mathbf{12}$     & $(RQ)^{({\frac{m}{2}-1)}}RP$             & $B^{\frac{m}{2}}A$  \\\hline
     4 & $\mathbf{0102}$   & $(PR)^{\frac{n-1}{2}}Q(RP)^{\frac{n-1}{2}}(QR)^{({\frac{m}{2}}-1)}Q$            & $C^{\frac{n+1}{2}}BC^{\frac{n+1}{2}}B^{\frac{m}{2}}$  \\\hline
     5 & $\mathbf{0212}$   & $(PR)^{({\frac{n}{2}}-1)}P(QR)^{\frac{m-1}{2}}P(RQ)^{\frac{m-1}{2}}$           & $C^{\frac{n}{2}}B^{\frac{m+1}{2}}CB^{\frac{m+1}{2}}$  \\\hline
     6 & $\mathbf{010212}$ & $(QR)^{\frac{m-1}{2}}P(RQ)^{\frac{m-1}{2}}(PR)^{\frac{n-1}{2}}Q(RP)^{\frac{n-1}{2}}$ & $B^{\frac{m+1}{2}}CB^{\frac{m+1}{2}}  C^{\frac{n+1}{2}}BC^{\frac{n+1}{2}}$  \\\hline
 \end{tabular}
\end{center}
\end{table}

\bigskip

The following result is useful in determining the link index.

\begin{lemma}
\label{lemmalinkindex}
 Let $T$ be a symmetry of a Riemann surface with mirror $M$.  Associated to $M$
 we have a pattern $\pi$ with link index $K$ and a mirror automorphism $S_M$.
 Then the order of $S_M$ is equal to $K$.
 \end{lemma}

\begin{proof}
We suppose that $M$ passes through vertices,
that is points labelled $\mathbf{0}$. If the link $\ell$ contains $\lambda>0$
vertices then the mirror contains $K\lambda$ vertices. Now if $\ell$ is a
link and $S_M$ has order $h$, then the pattern must be $\ell\cup S_M(\ell)\cup S_M^2(\ell)\dots\cup S_M^{h-1}(\ell)$ which contains $h\lambda$ vertices. Thus, $K\lambda=h\lambda$ so that $K=h$. If the mirror passes through face-centres or edge-centres, then the proof is the same.
\end{proof}

\section{Patterns of toroidal maps}
\setcounter{equation}{0}

It is known that a regular map of genus one is either of type
$\{4,4\}$, $\{3,6\}$ or $\{6,3\}$. To describe these maps we follow \cite{CM} and \cite{JS}.

\subsection{Maps of type $\{4,4\}$}

Consider the $(2,4,4)$-triangle $T$ in $\mathbb{C}$ whose vertices are $0$, $\frac{1}{2}$ and $\frac{1}{2}+\frac{1}{2}\mathrm{i}$. Then the reflections $P(z)=\bar{z}$, $Q(z)=-\bar{z}+1$ and $R(z)=\mathrm{i}\bar{z}$ in the sides of $T$ satisfy the relations
$$P^2=Q^2=R^2=(PQ)^2=(QR)^4=(RP)^4=1$$
and hence generate a group isomorphic to the extended triangle group $\Gamma(2,4,4)$.
Now $\Gamma(2,4,4)$ has a normal subgroup $\Lambda$ generated by the unit translations
$z\mapsto z+1$ and $z\mapsto z+\mathrm{i}$. The square with vertices $0$, $1$, $\mathrm{i}$
and $1+\mathrm{i}$ is a fundamental region for $\Lambda$ and $\Lambda$ can be
identified with the ring $\mathbb{Z}[\mathrm{i}]$ of Gaussian integers.
Now $\Lambda$ has a subgroup $\Omega$ generated by the perpendicular
translations $f(z)=z+b+c\mathrm{i}$ and $g(z)=z-c+b\mathrm{i}$, where $b$ and $c$ are integers.
Also, $\mathbb{C}/\Omega$ is a torus and $\Omega$ is an ideal in $\mathbb{Z}[\mathrm{i}]$ since $\mathrm{i}(b+c\mathrm{i})=-c+b\mathrm{i}$. In fact, $\Omega$ is the principal ideal $(b+c\mathrm{i})$.
The square with vertices $0$, $b+c\mathrm{i}$, $-c+b\mathrm{i}$ and $b-c+(b+c)\mathrm{i}$ is a
fundamental region for $\Omega$ and has area $n=b^2+c^2$. So the index of $\Omega$ in
$\Lambda$ is $n$. On the torus $\mathbb{C}/\Omega$ there is a map of type $\{4,4\}$ having
$n$ vertices $2n$ edges and $n$ faces and it is denoted by $\{4,4\}_{b,c}$ in \cite{CM}.
In this paper we deal with reflexible regular maps and by \cite{CM} the map $\{4,4\}_{b,c}$
is reflexible if and only if $bc(b-c)=0$. So we have two types of reflexible maps, namely
$\{4,4\}_{b,0}$ and $\{4,4\}_{b,b}$.

Now from \cite{JS} a regular map of type $\{4,4\}_{b,0}$ has the form $\mathbb{Z}[\mathrm{i}]/(b)$.
It follows from Section 3 that we have three types of patterns, namely $(\mathbf{01})^{K_1}$, $(\mathbf{12})^{K_2}$ and $(\mathbf{02})^{K_3}$, where $K_1$, $K_2$ and $K_3$ are
positive integers. The mirror automorphisms corresponding to these patterns are $BC^{-1}(z)=z+1$, $B^{-1}C(z)=z+\mathrm{i}$ and $BAC(z)=z+1+\mathrm{i}$, respectively, where $A$, $B$, $C$ are defined after (\ref{presentation2mn}).
It is thus clear that in all these cases the order of the mirror automorphism is $b$.
Also, from \cite{JS} regular maps of type $\{4,4\}_{b,b}$ have the form $\mathbb{Z}[\mathrm{i}]/(b+b\mathrm{i})$.
Now $b$ or $b\mathrm{i}$  is not in the ideal generated by $b+b\mathrm{i}$.
However $2b=(1-\mathrm{i})(b+b\mathrm{i})$ is in the ideal generated by
$b+b\mathrm{i}$, and so is $2b\mathrm{i}$. Thus the order of the mirror
automorphisms corresponding to the patterns $(\mathbf{01})^{K_1}$ and $(\mathbf{12})^{K_2}$ is $2b$.
For a pattern of type $(\mathbf{02})^{K_3}$ the mirror automorphism is $BAC(z)=z+1+\mathrm{i}$.
In this case $b(1+\mathrm{i})$ is in the ideal generated by $b(1+\mathrm{i})$ so that the order of
the mirror automorphism is $b$.  Consequently by Lemma~\ref{lemmalinkindex} we have

\begin{theorem}
Let $\mathcal{M}_1$ and $\mathcal{M}_2$ be the regular maps $\{4,4\}_{b,0}$
and $\{4,4\}_{b,b}$ on the square torus, respectively. Then the pattern of a
mirror is either $(\mathbf{01})^b$,
$(\mathbf{02})^b$ or $(\mathbf{12})^b$ with respect to $\mathcal{M}_1$, or
$(\mathbf{01})^{2b}$, $(\mathbf{02})^b$ or $(\mathbf{12})^{2b}$ with
respect to $\mathcal{M}_2$.
\end{theorem}

We illustrate this theorem in Figures~\ref{Map4440} and \ref{Map4422}, where the bold lines
denote the mirrors with different patterns.

\begin{figure}[htb]
  \centering
  \includegraphics[width=0.35\textwidth]{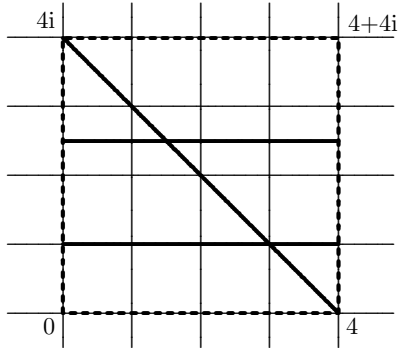}
  \caption{Map $\{4,4\}_{4,0}$}
  \label{Map4440}
\end{figure}

\begin{figure}[htb]
  \centering
  \includegraphics[width=0.35\textwidth]{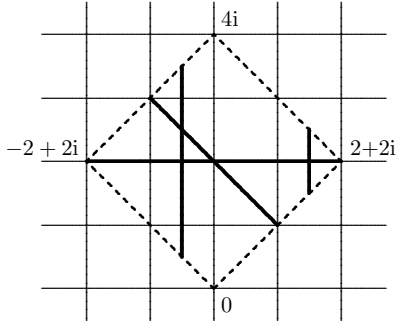}
  \caption{Map $\{4,4\}_{2,2}$}
  \label{Map4422}
\end{figure}

\subsection{Maps of types $\{3,6\}$ and $\{6,3\}$}

Now consider the $(2,3,6)$-triangle $T$ in $\mathbb{C}$ whose vertices are $0$,
$\frac{1}{2}$ and $\frac{1}{2}+\frac{\sqrt{3}}{6}\mathrm{i}$. Then the reflections
$P(z)=\bar{z}$, $Q(z)=-\bar{z}+1$ and $R(z)=a\bar{z}$,
where $a=\mathrm{e}^{\frac{\pi}{3}\mathrm{i}}$, in the sides of $T$ generate a group
isomorphic to the extended triangle group $\Gamma(2,3,6)$, which has a presentation
$$\langle P,Q,R\mid P^2=Q^2=R^2=(PQ)^2=(QR)^3=(RP)^6=1\rangle.$$
Then $\Gamma(2,3,6)$ has a normal subgroup $\Lambda$ generated by
the unit translations $z\mapsto z+1$ and $z\mapsto z+\omega$,
where $\omega=\mathrm{e}^{\frac{2\pi}{3}\mathrm{i}}$. The
parallelogram with vertices $0$, $1$, $\omega$ and $1+\omega$ is a
fundamental region for $\Lambda$ and we can think of $\Lambda$ as
the ring $\mathbb{Z}[\omega]$. The lattice $\Lambda$ has a
subgroup $\Omega$ generated by the translations $f(z)=z-c+b\omega$
and $g(z)=z+b+c+c\omega$, where $b$ and $c$ are integers. So
$\mathbb{C}/\Omega$ is a torus and $\Omega$ is an ideal in
$\mathbb{Z}[\omega]$ since $\omega(b+c+c\omega)=-c+b\omega$.  In
fact $\Omega$ is the principal ideal $(-c+b\omega)$. The
parallelogram with vertices $0$, $-c+b\omega$, $b+c+c\omega$ and
$b+(b+c)\omega$ is a fundamental region for $\Omega$. On the torus
$\mathbb{C}/\Omega$ we have a new map of type $\{3,6\}$ having $t$
vertices $3t$ edges and $2t$ triangular faces and it is denoted by
$\{3,6\}_{b,c}$ in \cite{CM}, where $t=b^2+bc+c^2$. As in the
previous case the map $\{3,6\}_{b,c}$ is reflexible if and only if
$bc(b-c)=0$ by \cite{CM}. So in this case we have the reflexible
regular maps $\{3,6\}_{b,0}$ and $\{3,6\}_{b,b}$.

Now a regular map of type $\{3,6\}_{b,0}$ has the form $\mathbb{Z}[\omega]/(b)$.
It follows from \hbox{Section 3} that in this case we have the patterns $(\mathbf{01})^{K_1}$ and $(\mathbf{0212})^{K_2}$ and the corresponding mirror automorphisms are $C^2A(z)=z-1$ and $C^3B^2CB^2(z)=z-2-\omega$, respectively.
It is thus clear that in both cases the order of the mirror automorphism is $b$.
Similarly, regular maps of type $\{3,6\}_{b,b}$ have the form $\mathbb{Z}[\omega]/(-b+b\omega)$.
Now $b$, $-b$, $2b$ and $-2b$ is not in the ideal generated by $-b+b\omega$. However, $-3b=(2+\omega)(-b+b\omega)$ is in the ideal generated by $-b+b\omega$ and so is $3b$.
Thus, the order of the mirror automorphism $C^2A$, which corresponds to the pattern
$(\mathbf{01})^{K_1}$, is $3b$. Since $b(-2-\omega)=(-b+b\omega)(1+\omega)$, $-2b-b\omega$ is in the ideal generated by $-b+b\omega$. So the mirror automorphism $C^3B^2CB^2$, which corresponds to the pattern
$(\mathbf{0212})^{K_2}$, has order $b$. As a result by Lemma~\ref{lemmalinkindex} we have

\begin{theorem}
Let $\mathcal{M}_1$ and $\mathcal{M}_2$ be the regular maps $\{3,6\}_{b,0}$ and $\{3,6\}_{b,b}$
on the rhombic torus, respectively. Then the pattern of a mirror is either $(\mathbf{01})^b$ or $(\mathbf{0212})^b$ with respect to $\mathcal{M}_1$, or $(\mathbf{01})^{3b}$ or $(\mathbf{0212})^b$ with respect to $\mathcal{M}_2$.
\end{theorem}

See Figures~\ref{Map3640} and \ref{Map3622} for illustrations of this theorem.

The regular maps $\{6,3\}_{b,0}$ and $\{6,3\}_{b,b}$ are the duals of the maps $\{3,6\}_{b,0}$ and $\{3,6\}_{b,b}$, respectively. The patterns of the mirrors with respect to them can be obtained from
the above theorem by changing the $\mathbf{0}$s and $\mathbf{2}$s.

\begin{figure}[htb]
  \centering
  \includegraphics[width=0.5\textwidth]{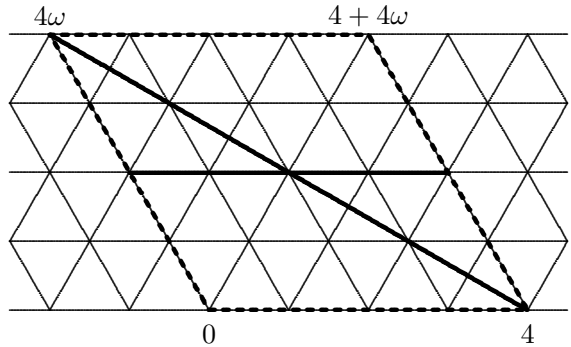}
  \caption{Map $\{3,6\}_{4,0}$}
  \label{Map3640}
\end{figure}

\begin{figure}[htb]
  \centering
  \includegraphics[width=0.6\textwidth]{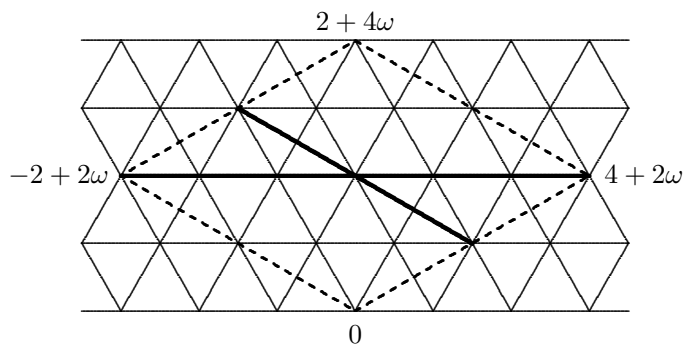}
  \caption{Map $\{3,6\}_{2,2}$}
  \label{Map3622}
\end{figure}

\section{Patterns of mirrors on surfaces of genus $g>1$}
\setcounter{equation}{0}

We now discuss patterns for regular maps on surfaces of genus $g>1$. The actual pattern
for a map of type $\{m,n\}$ is the same as for the universal map of type $\{m,n\}$ except that now the link indices
are finite. For example if $m$ and $n$ are odd our
pattern would be $(\mathbf{010212})^k$, for some integer $k$. Thus, for the icosahedron
of type $\{3,5\}$ we have $(\mathbf{010212})^2$ and for Klein's map of type $\{3,7\}$ we have
$(\mathbf{010212})^3$ as we have seen. A major problem is to compute these link indices and for this we use Lemma~\ref{lemmalinkindex}.

\subsection{Patterns of mirrors on surfaces of genus 2 and 3}

We now determine the patterns of mirrors on surfaces of genus 2 and 3.
For the presentations of the full automorphism groups of the corresponding
regular maps we use the list of \cite{CD}, where a presentation is given
for the full conformal automorphism groups of regular maps of genus 2 to 15.
Where we have a nice representation of the group we can find the order of the
mirror automorphism easily as illustrated in Example~\ref{exampleBolza}.
In other cases we could use MAGMA~\cite{MAGMA}.
We give the results in Tables~\ref{patterngenus2} and
\ref{patterngenus3}. Some of the maps in Tables~\ref{patterngenus2} and
\ref{patterngenus3} are particular examples of Accola-Maclachlan or Wiman maps which we deal with in \S 6.1 and 6.2.

\begin{example}
\label{exampleBolza} \emph{Let $X$ be the Riemann surface of genus 2 with the largest
automorphism group. In this case $X$ is the Bolza surface with 48 automorphisms.
It underlies a regular map $\mathcal{M}$ of type $\{3,8\}$
and we determine the patterns of the mirrors fixed by the reflections of $\mathcal{M}$.
Now $Aut\mathcal{M}$ is isomorphic to $\mathrm{GL}(2,3)$ and can be generated by
$
A=\left(
  \begin{array}{rr}
    1 & 1 \\
    0 & -1 \\
  \end{array}
\right)
$
 and
$
B=\left(
  \begin{array}{rr}
    0 & -1 \\
    1 & -1 \\
  \end{array}
\right).
$
 So we have
\[
C=B^{-1}A=\left(
  \begin{array}{rr}
    -1 & 1 \\
    -1 & 0 \\
  \end{array}
\right)
\left(
  \begin{array}{rr}
    1 & 1 \\
    0 & -1 \\
  \end{array}
\right)
=
\left(
  \begin{array}{rr}
    -1 & 1 \\
    -1 & -1 \\
  \end{array}
\right).
\]
Let $M$ be a mirror on $X$ fixed by a reflection of $\mathcal{M}$. It follows from
Section 3 that the pattern of $M$ is either $(\mathbf{0212})^{K_1}$ or $(\mathbf{01})^{K_2}$,
where $K_1$ and $K_2$ are positive integers. Let $M$ have pattern $(\mathbf{0212})^{K_1}$. From Table~\ref{tablemirroraut} it follows that the
mirror automorphism of $M$ is $C^{\frac{n}{2}}B^{\frac{m+1}{2}}CB^{\frac{m+1}{2}}$, where $m=3$ and $n=8$. So the mirror automorphism of $M$ is $C^4B^2CB^2$.
By using the above matrices we find that
\[
C^4B^2CB^2=\left(
  \begin{array}{rr}
    1 & 0 \\
    0 & -1 \\
  \end{array}
\right).
\]
This matrix has order 2 and hence by Lemma~\ref{lemmalinkindex}, $K_1=2$ and $M$ has pattern $(\mathbf{0212})^2$.
Now let $M$ has pattern $(\mathbf{01})^{K_2}$. Then by Table~\ref{tablemirroraut} the
mirror automorphism of $M$ is $C^{\frac{n}{2}}A$ and so the mirror automorphism of $M$
is $C^4A$. By using the above matrices we find that
\[
C^4A=\left(
  \begin{array}{rr}
    -1 & -1 \\
    0 & 1 \\
  \end{array}
\right),
\]
which has order 2. So by Lemma~\ref{lemmalinkindex}, $K_2=2$ and $M$ has pattern $(\mathbf{01})^2$.}
\end{example}

In some cases different maps may lie on the same Riemann surface.
For example the Bolza surface underlying the regular map of type $\{3,8\}$ also
underlies regular maps of types $\{4,8\}$ and $\{8,8\}$. This surface is denoted by $S1$ in Table~\ref{patterngenus2}, see \cite{Si2001a}, p63.

\begin{table}[H]
\begin{center}
\caption{\label{patterngenus2}Patterns of mirrors on surfaces of genus 2}
%\smallskip
\begin{tabular}{|c|c|c|c|c|c|}
   \hline
Map & Type & Surface  & $|Aut^\pm\mathcal M|$ &  Link & Link index \\
   \hline\hline

\cline{1-6}  &   &  &  &   $\mathbf{01}$ & 2\\
\cline{5-6}  \raisebox{1.5ex}[0pt]{M.2.1} &\raisebox{1.5ex}[0pt]{\{3,8\}}& \raisebox{1.5ex}[0pt]{S1} & \raisebox{1.5ex}[0pt]{96}  &  $\mathbf{0212}$ & 2\\
 \hline

\cline{1-6}    &  &   &  &    $\mathbf{01}$ & 4\\
\cline{5-6}  M.2.2 & \{4,6\}  & S2 & 48 &  $\mathbf{02}$ & 2\\
\cline{5-6}     &  &  &   &    $\mathbf{12}$ & 2\\
\hline

\cline{1-6}    &  &    & &    $\mathbf{01}$ & 2\\
\cline{5-6}  M.2.3 & \{4,8\}  & S1 & 32 &  $\mathbf{02}$ & 1\\
\cline{5-6}     &  &  &    &   $\mathbf{12}$ & 2\\
\hline

\cline{1-6}    &  &    & &    $\mathbf{01}$ & 2\\
\cline{5-6}  M.2.4 & \{6,6\}  & S2 & 24 &  $\mathbf{02}$ & 2\\
\cline{5-6}     &  &  &   &    $\mathbf{12}$ & 2\\
\hline

\cline{1-6}  &    & &  &   $\mathbf{01}$ & 1\\
\cline{5-6}  \raisebox{1.5ex}[0pt]{M.2.5} &\raisebox{1.5ex}[0pt]{\{5,10\}}& \raisebox{1.5ex}[0pt]{S3} & \raisebox{1.5ex}[0pt]{20}  &  $\mathbf{0212}$ & 1\\
 \hline

\cline{1-6}    &  &    & &    $\mathbf{01}$ & 1\\
\cline{5-6}  M.2.6 & \{8,8\}  & S1 & 16 &  $\mathbf{02}$ & 1\\
\cline{5-6}     &  &  &    &   $\mathbf{12}$ & 1\\
\hline

\end{tabular}
\end{center}
\end{table}

\begin{table}[H]
\begin{center}
\caption{\label{patterngenus3}Patterns of mirrors on surfaces of genus 3}
%\smallskip
\begin{tabular}{|c|c|c|c|c|c|}
   \hline
Map & Type & Surface  & $|Aut^\pm\mathcal M|$ &  Link & Link index \\
   \hline\hline

\cline{1-6}  M.3.1 & \{3,7\}  & S1 & 336 &  $\mathbf{010212}$ & 3\\

\cline{1-6}  &   &  &  &   $\mathbf{01}$ & 4\\
\cline{5-6}  \raisebox{1.5ex}[0pt]{M.3.2} &\raisebox{1.5ex}[0pt]{\{3,8\}}& \raisebox{1.5ex}[0pt]{S2} & \raisebox{1.5ex}[0pt]{192}  &  $\mathbf{0212}$ & 2\\
 \hline

\cline{1-6}  &   &  &  &   $\mathbf{01}$ & 2\\
\cline{5-6}  \raisebox{1.5ex}[0pt]{M.3.3} &\raisebox{1.5ex}[0pt]{\{3,12\}}& \raisebox{1.5ex}[0pt]{S3} & \raisebox{1.5ex}[0pt]{96}  &  $\mathbf{0212}$ & 2\\
 \hline

\cline{1-6}    &  &   &  &    $\mathbf{01}$ & 2\\
\cline{5-6}  M.3.4 & \{4,6\}  & S4 & 96 &  $\mathbf{02}$ & 2\\
\cline{5-6}     &  &  &   &    $\mathbf{12}$ & 4\\
\hline

\cline{1-6}    &  &    & &    $\mathbf{01}$ & 2\\
\cline{5-6}  M.3.5 & \{4,8\}  & S5 & 64 &  $\mathbf{02}$ & 2\\
\cline{5-6}     &  &  &    &   $\mathbf{12}$ & 2\\
\hline

\cline{1-6}    &  &    & &    $\mathbf{01}$ & 2\\
\cline{5-6}  M.3.6 & \{4,8\}  & S2 & 64 &  $\mathbf{02}$ & 2\\
\cline{5-6}     &  &  &    &   $\mathbf{12}$ & 4\\
\hline

\cline{1-6}    &  &    & &    $\mathbf{01}$ & 2\\
\cline{5-6}  M.3.7 & \{4,12\}  & S6 & 48 &  $\mathbf{02}$ & 1\\
\cline{5-6}     &  &  &    &   $\mathbf{12}$ & 2\\
\hline

\cline{1-6}    &  &    & &    $\mathbf{01}$ & 2\\
\cline{5-6}  M.3.8 & \{6,6\}  & S4 & 48 &  $\mathbf{02}$ & 1\\
\cline{5-6}     &  &  &    &   $\mathbf{12}$ & 2\\
\hline

\cline{1-6}    &  &    & &    $\mathbf{01}$ & 2\\
\cline{5-6}  M.3.9 & \{8,8\}  & S2 & 32 &  $\mathbf{02}$ & 1\\
\cline{5-6}     &  &  &    &   $\mathbf{12}$ & 2\\
\hline

\cline{1-6}    &  &    & &    $\mathbf{01}$ & 2\\
\cline{5-6}  M.3.10 & \{8,8\}  & S5 & 32 &  $\mathbf{02}$ & 1\\
\cline{5-6}     &  &  &    &   $\mathbf{12}$ & 2\\
\hline

\cline{1-6}  &    & &  &   $\mathbf{01}$ & 1\\
\cline{5-6}  \raisebox{1.5ex}[0pt]{M.3.11} &\raisebox{1.5ex}[0pt]{\{7,14\}}& \raisebox{1.5ex}[0pt]{S7} & \raisebox{1.5ex}[0pt]{28}  &  $\mathbf{0212}$ & 1\\
 \hline

\cline{1-6}    &  &    & &    $\mathbf{01}$ & 1\\
\cline{5-6}  M.3.12 & \{12,12\}  & S6 & 24 &  $\mathbf{02}$ & 1\\
\cline{5-6}     &  &  &    &   $\mathbf{12}$ & 1\\
\hline
\end{tabular}
\end{center}
\end{table}

\section{Patterns of mirrors on some well-known Riemann surfaces}
\setcounter{equation}{0}

\subsection{Accola-Maclachlan surfaces}

Let $\mu(g)$ be the maximum number of conformal
automorphisms of all Riemann surfaces of genus $g>1$. It was shown
independently by Accola \cite{Acc} and Maclachlan \cite{Macl} that
$\mu(g)\geq 8(g+1)$ and for every $g\geq 2$ there is a Riemann
surface of genus $g$ with $8(g+1)$ conformal automorphisms. These
surfaces are known as \emph{Accola-Maclachlan surfaces}. The Accola-Maclachlan
surface of genus $g$ is the two-sheeted covering of the sphere branched over
the vertices of a regular $(2g+2)$-gon.  Let
$X=\mathbb{H}/K$ be the Accola-Maclachlan surface of genus
$g>1$. It is known that $K$ is normal in the Fuchsian
triangle group $\Gamma[2,4,2g+2]$ and $AutX$ is isomorphic to
$\Gamma[2,4,2g+2]/K$ and has order $8(g+1)$. So $X$ underlies a
regular map $\mathcal{M}$ of type $\{2g+2,4\}$ such that
$Aut\mathcal{M}$ is isomorphic to $AutX$ and has a
presentation
\begin{equation}
\label{presentationAcMac+} \langle A,B,C\mid
A^2=B^{2g+2}=C^4=ABC=(C^{-1}B)^2=1\rangle
\end{equation}
This map is called the \emph{Accola-Maclachlan map}. The following
presentation of $Aut^\pm\mathcal{M}$ can be obtained from
(\ref{presentationAcMac+}):
\begin{equation}
\label{presentationAcMacFull} \langle P,Q,R\mid
P^2=Q^2=R^2=(PQ)^2=(QR)^{2g+2}=(RP)^4=(PRQR)^2=1\rangle.
\end{equation}
See \cite{Acc,Macl,Si2014} for details.

\begin{proposition}
\label{propositionAM} Let $\mathcal{M}$ be the map of type $\{2g+2,4\}$ on the the Accola-Maclachlan surface $X$ of genus $g>1$. Then the patterns of the mirrors of the reflections of $\mathcal{M}$
are either $(\mathbf{01})^2$, $(\mathbf{02})^2$ and $(\mathbf{12})^2$  if $g$ is odd or $(\mathbf{01})^2$, $(\mathbf{02})^2$ and $(\mathbf{12})^4$ if $g$ is even.
\end{proposition}

\begin{proof}
Let $M$ be a mirror of a reflection of $\mathcal{M}$. Then by
Theorem~\ref{theorempattern}, the pattern of $M$ is either
$(\mathbf{01})^{K_1}$, $(\mathbf{02})^{K_2}$ or
$(\mathbf{12})^{K_3}$, where ${K_1}$, ${K_2}$ and ${K_3}$ are
positive integers. Let $M$ have pattern $(\mathbf{01})^{K_1}$. Then by
Table~\ref{tablemirroraut}, the
mirror automorphism of $M$ is $C^\frac{n}{2}A$. Here $n=4$ and so the mirror
automorphism of $M$ is $C^2A$. It follows from
(\ref{presentationAcMac+}) that $C^2A=CB^{-1}$ and this element has order 2.
Thus, by Lemma~\ref{lemmalinkindex}, $M$ has pattern $(\mathbf{01})^2$.

Let $M$ have pattern $(\mathbf{02})^{K_2}$. It follows from
Table~\ref{tablemirroraut} that the
mirror automorphism of $M$ is $B^\frac{m}{2}C^\frac{n}{2}$,
where $m=2g+2$ and $n=4$. So the mirror automorphism of $M$ is $B^{g+1}C^2$.
It is known that $X$ is hyperelliptic and $C^2$ is the hyperelliptic involution, which is
central in $AutX$. Now $B^{g+1}C^2$ cannot be the identity. To see this let
$D_{2g+2}=\langle p,q\mid p^2=q^{2g+2}=(pq)^2=1\rangle$ be the dihedral group of order $4g+4$.
Then there is a homomorphism $\theta\colon AutX\rightarrow D_{2g+2}$ defined by
$\theta(B)=q$, $\theta(C)=p$. If $B^{g+1}C^2$ is the identity,
then $q^{g+1}p^2$ is the identity. So $q^{g+1}$ is the identity, which is false in $D_{2g+2}$.
As $C^2$ is central, $B^{g+1}C^2$ has order two and by
Lemma~\ref{lemmalinkindex}, $M$ has pattern $(\mathbf{02})^2$.

Now let $M$ have pattern $(\mathbf{12})^{K_3}$. According to
Table~\ref{tablemirroraut}, the mirror automorphism of $M$ is $B^\frac{m}{2}A$,
where $m=2g+2$. So $B^{g+1}A$ is the mirror automorphism of $M$ and
it cannot be the identity. Otherwise, $A=B^{g+1}$ and $AutX$ is cyclic, which is false.
Now by using (\ref{presentationAcMac+}) it can be shown that for every $g>1$, $C^{-1}B^gC^{-1}$
is equal to $B^{-g}$ if $g$ is odd and $B^{-g}C^2$ if $g$ is even. So
$(B^{g+1}A)^2=B^gC^{-1}B^gC^{-1}$ is the identity if $g$ is odd and $C^2$ if $g$ is even.
It follows that the mirror automorphism $B^{g+1}A$ has
order 2 if $g$ is odd and 4 if $g$ is even. Thus, by Lemma~\ref{lemmalinkindex}, $M$ has pattern $(\mathbf{12})^2$ if $g$ is odd and $(\mathbf{12})^4$ if $g$ is even.
\end{proof}

\subsection{Wiman surfaces}

According to a classical theorem of Wiman \cite{Wim}, the largest
possible order of an automorphism of a Riemann surface of genus
$g>1$ is $4g+2$ and the second largest possible order is $4g$.  According to \cite{Kul1},
there is a unique Riemann surface of genus $g$ admitting an automorphism of order $4g+2$.
Also, see Harvey \cite{Harv}. This is called the \emph{Wiman surface of type I} and it has an
underlying regular map $\mathcal{M}_1$ of type $\{4g+2,2g+1\}$ called the \emph{Wiman map of type I}.
Also for $g\not=3$ there is a unique surface of genus $g$ admitting an automorphism of
order $4g$, called the \emph{Wiman surface of type II} and it has an underlying regular map
$\mathcal{M}_2$ of type $\{4g,4\}$  called the \emph{Wiman map of type II}. This is because,
due to Harvey's work these come from epimorphisms from $\Gamma[2,2g+1, 4g+2]$ to $C_{4g+2}$
and from $\Gamma[2,4,4g]$ to $Aut\mathcal{M}_2$ respectively. (For $g=3$, there is
another Riemann surface admitting $12=4g$ automorphisms and this
comes from the epimorphism from $\Gamma[3,4,12]$ to $C_{12}$.)

For every $g>2$, the group $Aut\mathcal{M}_2$
has order $8g$ and has a presentation
\begin{equation}
\label{presentationWimanII} \langle A,B,C \mid A^2=B^{4g}=C^4=ABC=C^2B^{2g}=1\rangle,
\end{equation}
which can be deduced from \cite{Kul2}. The following presentation of $Aut^\pm\mathcal{M}_2$
is obtained from (\ref{presentationWimanII}):
\begin{equation}
\label{presentationWimanIIfull} \langle P,Q,R\mid
P^2=Q^2=R^2=(PQ)^2=(QR)^{4g}=(RP)^4=(RP)^2(QR)^{2g}=1\rangle.
\end{equation}
If $g=2$, then the kernel of the homomorphism from $\Gamma[2,8,8]$ onto $C_8$ can be shown to be normal  in $\Gamma[2,3,8]$ and so for genus 2 the Wiman surface of type II coincides with the Bolza surface of Example~\ref{exampleBolza}.

It is known that the groups $Aut\mathcal{M}_1$ and
$Aut^\pm\mathcal{M}_1$ are isomorphic to $C_{4g+2}$ and
$D_{4g+2}$, respectively.

\begin{proposition}
\label{propositionWimanI}Let $\mathcal{M}_1$ be the Wiman map of type I of genus $g>1$.
Then the pattern of a mirror is either $(\mathbf{12})^1$ or $(\mathbf{0102})^1$.
\end{proposition}

\begin{proof}
It is known that $\mathcal{M}_1$ has type $\{4g+2,2g+1\}$ and it has one
face and two vertices. So the pattern of a mirror have to be either
$(\mathbf{12})^1$ or $(\mathbf{0102})^1$ with respect to $\mathcal{M}_1$.
\end{proof}

\begin{proposition}
\label{propositionWimanII}Let $\mathcal{M}_2$ be the Wiman map of type II
of genus $g>2$. Then the pattern
of a mirror is either $(\mathbf{01})^2$, $(\mathbf{12})^2$ or $(\mathbf{02})^1$.
\end{proposition}

The proof is much the same as that of Proposition~\ref{propositionAM} so we omit it.

\subsection{Fermat curves}
\setcounter{equation}{0}

The Fermat curve $\mathbf{F}_n$ is the Riemann surface of the
projective algebraic curve
\[
\{(x,y,z)\mid x^n+y^n+z^n=0\},
\]
in $\mathbf{P}^2(\mathbb{C})$.
Let $\Gamma[n,n,n]$ be the Fuchsian triangle group with the presentation
\[
\langle x,y,z\mid x^n+y^n+z^n=xyz=1\rangle.
\]
If we abelianize this group, we get $\mathbb{Z}_n\oplus\mathbb{Z}_n$ so
that the commutator subgroup $K_n$ of $\Gamma[n,n,n]$ is a characteristic subgroup of index $n^2$.
So there is a homomorphism $\theta\colon \Gamma[n,n,n]\to \mathbb{Z}_n\oplus\mathbb{Z}_n$
whose kernel is $K_n$ and $\mathbb{H}/K_n\cong \mathbf{F}_n$, \cite{JS2}. Now
$\Gamma[n,n,n]\lhd \Gamma[2,3,2n]$ (\cite{Si1972}) with $\Gamma[2,3,2n]/\Gamma[n,n,n]\cong S_3$ (think of $S_3=\Gamma[2,3,2]$). We can extend the homomorphism $\theta$ to
$\theta^*\colon \Gamma[2,3,2n]\to (\mathbb{Z}_n\oplus\mathbb{Z}_n)\rtimes S_3$ so the kernel of
$\theta^*$ is also $K_n$. Thus, $K_n$ is normal in $\Gamma[2,3,2n]$ with index $6n^2$ and hence there is a regular map of type $\{3,2n\}$ on $\mathbf{F}_n$. We then find from the Riemann-Hurwitz formula that
$\mathbf{F}_n$ has genus $(n-1)(n-2)/2$. As $K_n$ is characteristic in
$\Gamma[n,n,n]$ and as $\Gamma[n,n,n]$ is normal in $\Gamma[2,3,2n]$,
$K_n$ is normal in $\Gamma[2,3,2n]$ and $Aut\mathbf{F}_n\cong \Gamma[2,3,2n]/K_n\cong (\mathbb{Z}_n\oplus\mathbb{Z}_n)\rtimes S_3$.

Consider the following matrices
$
A=\left(
  \begin{array}{lll}
    0            & \lambda & 0 \\
    \lambda^{-1} & 0       & 0 \\
    0            & 0       & 1 \\
  \end{array}
\right),
$
$
B=\left(
  \begin{array}{lll}
    0            & \lambda & 0 \\
    0            & 0       & 1 \\
    1            & 0       & 0 \\
  \end{array}
\right)
$
and
$
C=\left(
  \begin{array}{lll}
    0             & 0  & 1 \\
    0             & 1  & 0 \\
    \lambda^{-1}  & 0  & 0 \\
  \end{array}
\right) $ $\in \mathrm{PGL}(3,\mathbb{C})$, where
$\lambda=\mathrm{e}^{\frac{2\pi}{n}\mathrm{i}}$. Clearly, $A$, $B$
and $C$ are automorphisms of the Fermat curve $\mathbf{F}_n$. Also
$A^2=B^3=C^{2n}=ABC=1$. The matrices $C^2$ and $AC^2A$ are of
order $n$ and we can see that they commute. So they generate the
$\mathbb{Z}_n\oplus\mathbb{Z}_n$ part of $Aut\mathbf{F}_n$. As
$A^2=B^3=1$, the group generated by $A,B,C$ has order $6n^2$, so
$A,B,C$ generate $Aut\mathbf{F}_n$.

We now determine the patterns of the mirrors on $\mathbf{F}_n$ with respect to $\mathcal{M}_n$.
It follows from Section 3 that we have patterns $(\mathbf{0212})^{K_1}$ and
$(\mathbf{01})^{K_2}$ and the corresponding mirror automorphisms are
$S_1=C^nB^2CB^2$ and $S_2=C^nA$, respectively. We find that if $n$ is odd, then
$C^n=\left(
  \begin{array}{lll}
    0                      & 0  & \lambda^\frac{-n+1}{2} \\
    0                      & 1  & 0 \\
   \lambda^\frac{-n-1}{2}  & 0  & 0 \\
  \end{array}
\right)$
and if $n$ is even, then
$C^n=\left(
  \begin{array}{lll}
    \lambda^{-\frac{n}{2}} & 0  & 0 \\
    0                      & 1  & 0 \\
    0                      & 0  & \lambda^{-\frac{n}{2}} \\
  \end{array}
\right)$.
Hence if $n$ is odd, then
$S_1=C^nB^2CB^2=\left(
  \begin{array}{lll}
    \lambda^\frac{-n+3}{2}  & 0         & 0 \\
    0                       & \lambda   & 0 \\
    0                       & 0         & \lambda^\frac{-n+1}{2} \\
  \end{array}
\right)$
and if $n$ is even, then
$S_1=\left(
  \begin{array}{lll}
    0                        & 0        & \lambda^{-\frac{n}{2}+1} \\
    0                        & \lambda  & 0 \\
   \lambda^{-\frac{n}{2}+1}  & 0        & 0 \\
  \end{array}
\right)$.

Let $n$ be even. Then we find that
$S_1^2=\left(
  \begin{array}{lll}
    \lambda^2 & 0          & 0 \\
    0         & \lambda^2  & 0 \\
    0         & 0          & \lambda^2 \\
  \end{array}
\right)$,
which is the identity. So $S_1$ has order 2 and by
Lemma~\ref{lemmalinkindex}, we have pattern $(\mathbf{0212})^2$.

Let now $n$ be odd. Then in this case we can see that $S_1$ has order $n$. Therefore, by
Lemma~\ref{lemmalinkindex}, we have pattern $(\mathbf{0212})^n$.

Let us determine the order of $S_2=C^nA$. Now we
find that if $n$ is odd, then
$S_2=\left(
  \begin{array}{lll}
    0             & 0                         & \lambda^{\frac{-n+1}{2}} \\
    \lambda^{-1}  & 0                         & 0 \\
    0             & \lambda^{\frac{-n-1}{2}}  & 0 \\
  \end{array}
\right)$
and if $n$ is even, then
$S_2=\left(
  \begin{array}{lll}
    0             & \lambda^{-\frac{n}{2}+1}  & 0 \\
    \lambda^{-1}  & 0                         & 0 \\
    0             & 0                         & \lambda^{-\frac{n}{2}} \\
  \end{array}
\right)$.

We can see that $S_2$ has order 3 if $n$ is odd and 4 if $n$ is even. Thus, by
Lemma~\ref{lemmalinkindex}, we have pattern $(\mathbf{01})^3$ if $n$ is odd and
$(\mathbf{01})^4$ if $n$ is even.
As a result we have

\begin{theorem}
\label{theoremH1}
Let $\mathbf{F}_n$ be the Fermat curve of degree $n$ and consider the map $\mathcal{M}_n$ of type $\{3,2n\}$ on $\mathbf{F}_n$. Then:
\begin{enumerate}
 \item[$(i)$] If $n$ if odd, then the patterns of the mirrors on $\mathbf{F}_n$ have one of the forms
 $(\mathbf{01})^3$ and $(\mathbf{0212})^n$;
  \item[$(ii)$] If $n$ if even, then the patterns of the mirrors on $\mathbf{F}_n$ have one of the forms
 $(\mathbf{01})^4$ and $(\mathbf{0212})^2$.
\end{enumerate}
\end{theorem}

If we consider the dual map of $\mathcal{M}_n$, then we get the patterns above with $\mathbf{0}$s and $\mathbf{2}$s interchanged.

\subsection{Hurwitz surfaces}

Let $g>1$ be a positive integer and $\mu(g)$ be the maximum number
of conformal automorphisms of all Riemann surfaces of genus $g$.
Then it is known that $\mu(g)\leq 84(g-1)$ and this upper bound is
attained for infinitely many $g>1$, see \cite{Macb1}.  If
$X=\mathbb{H}/K$ is a Riemann surface of genus $g>1$ with
$|Aut X|= 84(g-1)$, then it is called a \emph{Hurwitz surface}
and $Aut X$ is called a \emph{Hurwitz group}. Here $Aut X\cong
\Gamma/K$, where $\Gamma$ is the Fuchsian triangle group
$\Gamma[2,3,7]$. Therefore, $X$ underlies a regular map
$\mathcal{M}$ of type $\{3,7\}$, which is called a \emph{Hurwitz
map}. It is also known that $Aut\mathcal{M}$ is isomorphic to
$AutX$ and has a presentation of the form
\begin{eqnarray*}
\langle A,B,C\mid A^2=B^3=C^7=ABC=\dots=1\rangle.
\label{presentationHurwitz}
\end{eqnarray*}
Let $\mathcal{M}$ be the corresponding Hurwitz map and let $S$ be the
mirror automorphism of $\mathcal{M}$. So $S=B^2CB^2C^4BC^4$
(see Table~\ref{tablemirroraut}). If $S$ has order 1, then
\[
AutX=\langle A,B,C\mid A^2=B^3=C^7=ABC=B^2CB^2C^4BC^4=1\rangle.
\]
Now by using MAGMA, we find that $AutX$ is the trivial group.
This shows that a Hurwitz map cannot have a link index equal to one.
Similarly,
\[
\langle A,B,C\mid A^2=B^3=C^7=ABC=S^2=1\rangle
\]
is a presentation for a group of order 504. So the genus must be 7
and $X$ is Macbeath's surface, \cite{Macb2}.
In the same way, we see that
\[
\langle A,B,C\mid A^2=B^3=C^7=ABC=S^3=1\rangle
\]
is a group of order 168. So $X$ is Klein's surface of genus 3.

We have proved the following:

\begin{theorem}
\label{theoremH1}
\begin{enumerate}
 \item[$(i)$] No Hurwitz map can have link index equal to $1$.
\item[$(ii)$] A Hurwitz map has link index $2$ if and only if the underlying Riemann
surface is Macbeath's surface of genus $7$.
\item[$(iii)$] A Hurwitz map has link index $3$ if and only if the underlying Riemann surface is Klein's surface of genus $3$.
\end{enumerate}
\end{theorem}

If the link index $K=4$, then MAGMA shows that the automorphism group has order
16515072 so the Hurwitz surface has genus 196609. If $K\geq 5$, then
MAGMA does not give an answer. This means that either no Hurwitz
surface with this link index exists, or possibly the computer is not yet
powerful enough! However, given a Hurwitz group of genus less than 302
we can use \cite{Cond} to find its presentation and then MAGMA will
tell us the order of the word $B^2CB^2C^4BC^4$ and thus we can compute
the link index. These results are displayed in Table~\ref{HurwitzValues}.

\begin{remark}
\label{remarkH1} \emph{It follows from Table~\ref{HurwitzValues} that
different Hurwitz maps can have the same link index.}
\end{remark}

Every mirror on a Hurwitz surface is a combination of the sides of
$(2,3,7)$-triangles. By using sine and cosine rules for hyperbolic
triangles we find that the sides of every $(2,3,7)$-triangle have approximate lengths
$a=0.2831281533$, $b=0.5452748317$ and $c=0.6206717375$. A link $\mathbf{010212}$
corresponds to a segment of a mirror, which consists of six triangle sides
and has length $2(a+b+c)=2.8981494452$. So a mirror of pattern $(\mathbf{010212})^K$
on a Hurwitz surface has length $2K(a+b+c)$. Hence by Theorem~\ref{theoremH1}, we have

\begin{theorem}
\label{theoremHurwitzlength} The length of a mirror on a Hurwitz
surface cannot be less than $5.7962988904$.  This lower bound is attined
if and only if the underlying Riemann surface is Macbeath's surface
of genus $7$. The second shortest mirror on a Hurwitz surface has length
$8.6944483356$. In this case the Riemann surface is Klein's surface of genus $3$.
\end{theorem}

\begin{table}[H]
\begin{center}
\caption{Link indices for some Hurwitz maps}
\label{HurwitzValues}
\smallskip
\begin{tabular}{|l|c|c|c|}
  \hline
 Map &  Genus & $|Aut\mathcal{M}|$ & Link index  \\
  \hline%\hline
  \cline{1-4} H1 & 3 & 168 & 3  \\ \hline
  \cline{1-4} H2 & 7 & 504 & 2  \\ \hline
\cline{1-4} H3 & &  & 7  \\
   \cline{1-1}  \cline{4-4} H4 &  14 &1092& 7       \\
    \cline{1-1} \cline{4-4} H5 &  &  & 6    \\ \hline
\cline{1-4} H6 & 118 & 9828 & 13 \\ \hline

\cline{1-4} H7 & 129 & 10752 &  6   \\
\hline

\cline{1-4} H8 & &  & 15  \\
  \cline{1-1}   \cline{4-4}  H9 & 146 &  12180  & 15     \\
  \cline{1-1}  \cline{4-4}  H10  &  &  & 14  \\ \hline
\end{tabular}
\end{center}
\end{table}

\bigskip
\bigskip

\textbf{Acknowledgements}

We would like to thank Marston Conder for help with the section on Hurwitz surfaces
and Gareth Jones for help in the section on Fermat curves.


\begin{thebibliography}{99}

\bibitem{Acc} R. D. M. Accola, On the number of automorphisms
of a closed Riemann surface, \emph{Trans. Amer. Math. Soc.} \textbf{131} (1968),
398--408.

\bibitem{MAGMA} W. Bosma, J. Cannon and C. Playoust, The Magma algebra system. I.
The user language. \emph{J. Symbolic Comput.} \textbf{24} (1997), 235--265.

\bibitem{Cond} M. Conder, \emph{All reflexible orientable regular maps on
surfaces of genus 2 to 301, up to isomorphism and duality, with defining
relations for their automorphism groups}, from https://www.math.auckland.ac.nz/~conder/OrientableRegularMaps301.txt

\bibitem{CD} M. Conder and P. Dobcs\'{a}nyi,
Determination of all Regular Maps of Small Genus,
\emph{J. Combin. Theory Ser. B} \textbf{81} (2001), 224--242.

\bibitem{Co} H. S. M. Coxeter, \emph{Regular Polytopes}, Dover Publications, New York, 1973.

\bibitem{CM} H. S. M. Coxeter and W. O. J. Moser, \emph{Generators and Relations
for Discrete Groups}, Springer-Verlag, 1980.

\bibitem{Grot}  A. Grothendieck, Esquisse d'un Programme,
\emph{Geometric Galois Actions} I, London Mathematical
Society Lecture Note Series 242, Cambridge University Press,
Cambridge, 1997, 5--48.

\bibitem{Harv} W. J. Harvey, Cyclic groups of automorphisms of compact
Riemann surfaces, \emph{Q. J. Math.} \textbf{17} (1966), 86--97.

\bibitem{JS} G. A. Jones and D. Singerman, Theory of maps on
orientable surfaces, \emph{Proc. Lond. Math. Soc.} (3)
\textbf{37} (1978), 273--307.

\bibitem{JS2} G. A. Jones and D. Singerman, Bely\u{\i} functions hypermaps and Galois groups,
\emph{Bull. Lond. Math. Soc.}
\textbf{28}(6) (1996), 561--590.

\bibitem{Kle} F. Klein, \"Uber die Transformation siebenter Ordnung der
elliptischen Funktionen, \emph{Math. Ann.} \textbf{14} (1879), 428--471.

\bibitem{Kul1} R. S. Kulkarni, A note on Wiman and Accola-Maclachlan
surfaces, \emph{Ann. Acad. Sci. Fenn. Math.} \textbf{16} (1991), 83--94.

\bibitem{Kul2} R. S. Kulkarni, \emph{Riemann surfaces admitting large automorphism
groups}, in \emph{Contemp. Math.} (Amer. Math. Soc., Providence,
RI, 1997) \textbf{201} pp. 63--79.

\bibitem{Levy} S. Levy, On the Order-Seven Transformation of Elliptic Functions,
\emph{The eightfold way; the beauty of Klein's quartic curve}, Math. Sci. Res. Inst. Publ. 35,
Cambridge University Press, 1999, 287--331.

\bibitem{Macb1} A. M. Macbeath, On a theorem of Hurwitz, \emph{Proc. Glasgow Math.
Assoc.} \textbf{5} (1961), 90--96.

\bibitem{Macb2} A. M. Macbeath, On a curve of genus 7,
\emph{Proc. Lond. Math. Soc.} (3) \textbf{15} (1965), 527--542.

\bibitem{Macl} C. Maclachlan, A bound for the number of automorphisms
of a compact Riemann surface, \emph{J. Lond. Math. Soc.} \textbf{44} (1969),
265--272.

\bibitem{Me} A. Meleko\u glu, A Geometric approach to the reflections
of regular maps, \emph{Ars Combin.} \textbf{89} (2008), 355--367.

\bibitem{MS} A. Meleko\u glu and D. Singerman, Reflections
of regular maps and Riemann surfaces, \emph{Rev. Mat. Iberoam.}
\textbf{24} (2008), 921--939.

\bibitem{Si1972} D. Singerman, Finitely maximal Fuchsian groups,
\emph{J. Lond. Math. Soc.} \textbf{6} (1972), 29--38.

\bibitem{Si1976} D. Singerman, Automorphisms of maps, permutation groups
and Riemann surfaces, \emph{Bull. Lond. Math. Soc.} \textbf{8} (1976), 65--68.

\bibitem{Si2001a} D. Singerman, Riemann surfaces, Belyi functions and hypermaps,
\emph{Topics on Riemann surfaces and Fuchsian groups}, London Mathematical
Society Lecture Note series 287, Cambridge University Press, Cambridge, 2001.

\bibitem{Si2014} D. Singerman, \emph{The remarkable Accola-Maclachlan surfaces},
in \emph{Contemp. Math.} (Amer. Math. Soc., Providence,
RI, 2014) \textbf{629} pp. 315--322.

\bibitem{Wim} A. Wiman, \emph{\"Uber die hyperelliptischen curven und
diejenigen vom geschlechte $p=3$ welche eindeutigen
transformationen in sich zulassen}. Bihang Till Kongl. Svenska
Veienskaps-Akademiens Handlingar, (Stockholm 1895-6), bd.21,
1--23.


\end{thebibliography}
\end{document}